\documentclass{article}

\usepackage{amsmath}
\usepackage{amssymb}
 \usepackage{blkarray}
\usepackage{amsfonts, amsbsy}
\usepackage{xspace}
\usepackage{graphicx}
\usepackage{algpseudocode} 
\usepackage{algorithm}
\usepackage{xcolor}
\usepackage{hyperref}
\usepackage{multirow}
\usepackage{authblk,tikz}
\usepackage[utf8]{inputenc}
\usepackage{enumerate}
\usepackage{fullpage}
\usepackage{verbatim}
\usepackage{amsmath,amsthm,amssymb}
\usepackage{amsfonts, amsbsy}
\usepackage{xspace}
\usepackage{graphicx}
\usepackage{xcolor}
\usepackage{hyperref}
\usepackage{authblk,tikz}
\usepackage{eso-pic}

\catcode`\=13
\def{$\bowtie$}

\usepackage{eurosym}
\usepackage{lscape} %landcape pages support

\newtheorem{prethm}{{\bf Theorem}}

\newenvironment{thm}{\begin{prethm}{\hspace{-0.5
				em}{\bf}}}{\end{prethm}}

\newtheorem{prepro}{{\bf Theorem}}

\newtheorem{preprop}{{\bf Proposition}}

\newenvironment{prop}{\begin{preprop}{\hspace{-0.5
				em}{\bf}}}{\end{preprop}}

\newtheorem{precor}{{\bf Corollary}}

\newenvironment{cor}{\begin{precor}{\hspace{-0.5
				em}{\bf}}}{\end{precor}}

\newtheorem{preconj}{{\bf Conjecture}}

\newenvironment{conj}{\begin{preconj}{\hspace{-0.5
				em}{\bf}}}{\end{preconj}}

\newtheorem{predefi}{{\bf Definition}}

\newtheorem{preremark}{{\bf Remark}}

\newtheorem{preexample}{{\bf Example}}

\newenvironment{example}{\begin{preexample}\rm{\hspace{-0.5
				em}{\bf}}}{\end{preexample}}

\newtheorem{prelem}{{\bf Lemma}}

\newenvironment{lem}{\begin{prelem}{\hspace{-0.5
				em}{\bf}}}{\end{prelem}}

\newtheorem{prelam}{{\bf Lemma}}

\newtheorem{preprob}{{\bf Problem}}

\newtheorem{preali}{{\bf Proof of Theorem 1.}}

\newtheorem{prealii}{{\bf Proof of Theorem 2.}}

\newtheorem{prealiii}{{\bf Proof of Theorem 3.}}

\newtheorem{prealiiii}{{\bf Proof of Theorem 4.}}

\newtheorem{prealij}{{\bf Proof of Theorem 5.}}

\newtheorem{prealijj}{{\bf Proof of Theorem 6.}}

\newtheorem{prealijjj}{{\bf Proof of Theorem 7.}}

\newtheorem{prealijjjk}{{\bf Proof of Theorem 8.}}

\newenvironment{subproof}{\par\noindent {\it Proof}.\ }{\hfill$\blacksquare$\par\vspace{11pt}}

\algnewcommand{\parState}[1]{\State
    \parbox[t]{\dimexpr\linewidth-\algmargin}{\strut #1\strut}}

\title{On the In-Out-Proper Orientations of Graphs}

\author[1]{Ali Dehghan}

\affil[1]{Systems and Computer Engineering Department, Carleton University, Ottawa,   Canada}

\begin{document}
\maketitle

\begin{abstract}
{\small \noindent
An orientation of a graph $G$ is {\it in-out-proper} if any two adjacent vertices have different in-out-degrees, where the in-out-degree of each vertex is equal to the in-degree minus the out-degree of that vertex. The {\it in-out-proper orientation number} of a graph $G$, denoted by   $\overleftrightarrow{\chi}(G)$, is  $ \min_{D\in \Gamma}\max_{v\in V(G)} |d_D^{\pm}(v)|$, where $\Gamma$ is the set of in-out-proper orientations of $G$ and $d_D^{\pm}(v)$ is the in-out-degree of the vertex $v$ in the orientation $D$. 
Borowiecki {\it et al.} proved that the
in-out-proper orientation number
is well-defined for any graph $G$ [Inform. Process. Lett., 112(1-2):1–4, 2012]. So we have 
$ \overleftrightarrow{\chi}(G) \leq \Delta(G)$, where $\Delta(G)$ is the maximum degree of vertices in $G$. We conjecture that there exists a constant number   $c$ such that for every planar graph $G$, we have $\overleftrightarrow{\chi}(G) \leq c $. Towards this speculation, we show that for every tree $T$ we have $\overleftrightarrow{\chi}(T) \leq 3 $ and this bound is sharp. Next, we study the in-out-proper orientation number of subcubic graphs. By using the properties of totally unimodular matrices
we show that there is a polynomial time algorithm to determine whether $\overleftrightarrow{\chi}(G) \leq 2$, for a given graph $G$ with maximum degree three. On the other hand, we show that it is NP-complete to decide whether $\overleftrightarrow{\chi}(G) \leq 1$ for a given bipartite graph $G$ with maximum degree three. Finally, we study  the in-out-proper orientation number of  regular graphs.
}

\begin{flushleft}

\end{flushleft}

\end{abstract}

\section{Introduction}
\label{}

Let $G$ be a graph and $D$ be an orientation of it. For every vertex $v$ of $G$, we denote the in-degree (out-degree) of $v$ in the orientation $D$ by $d_D^-(v)$ ($d_D^+(v)$, respectively).  
An orientation of a graph $G$ is called {\it proper} if any two adjacent vertices have different in-degrees \cite{MR3095464}. The {\it proper
orientation number} of a graph $G$, denoted by $\overrightarrow{\chi} (G)$,  is the minimum of the maximum in-degree taken over all proper orientations of the graph $G$. A proper orientation $D$ of $G$ can be used to form a proper vertex coloring of $G$ by assigning every vertex $v$ of $G$ the color $d_D^-(v)$ \cite{MR3095464}.  So, we have

\begin{equation}\label{E1}
 \chi(G)-1 \leq  \overrightarrow{\chi} (G)\leq \Delta(G).
\end{equation}

The proper orientation number of graphs has been studied by several authors, for instance see \cite{MR3095464, MR3704829,
MR4158396, 
MR3293286,
MR3514378,
MR3958387,
Dehghan,
semi,
   MR3714524, MR4092623}. In \cite{MR3293286}, Araujo {\it et al.} asked whether the proper orientation number of a planar graph is bounded. Toward this question, it was shown that if $T$ is a tree, then $\overrightarrow{\chi} (T)\leq 4$ \cite{MR3293286}. Also, it was shown that 
every cactus admits a proper orientation with maximum in-degree at most 7 \cite{MR3514378}. Furthermore, it was proved  that every bipartite planar graph with minimum degree at least 3 has proper orientation number at most 3 \cite{MR4092623}.

Let $D$ be an orientation for a given graph $G$. 
The in-out-degree of the vertex $v$ is defined as $ d_D^{\pm}(v)=d_D^-(v) - d_D^+(v)$. Note that for  a given graph $G$ and orientation $D$,  for each vertex $v$ we have

\begin{equation}\label{E2}
 -\Delta(G) \leq  d_D^{\pm}(v) \leq \Delta(G).
\end{equation}

Motivated by the proper orientations of graphs we investigate the in-out-proper orientations. 
An orientation of a graph $G$ is {\it in-out-proper} if any two adjacent vertices have different in-out-degrees.
The {\it in-out-proper orientation number} of a graph $G$, denoted by   $\overleftrightarrow{\chi}(G)$, is  $ \min_{D\in \Gamma}\max_{v\in V(G)} |d_D^{\pm}(v)|$, where $\Gamma$ is the set of in-out-proper orientations of $G$ and $d_D^{\pm}(v)$ is the in-out-degree of the vertex $v$ in the orientation $D$.
For a given graph $G$, we say that an in-out-proper orientation $D$ is {\it optimal} if the maximum of the absolute values of their in-out-degrees is equal to $\overleftrightarrow{\chi}(G)$.

It is interesting to mention that in-out-proper orientation relates to the flow. In more details, an in-out-proper orientation of a graph $G$ can be thought as a ‘flow’ of $G$ that does not satisfy Kirchhoff’s Current Law.
Borowiecki {\it et al.} proved that 
in-out-proper orientation number
is well-defined for any graph $G$ \cite{MR2895496}. 

\begin{thm}\label{T1} \cite{MR2895496}
The in-out-proper orientation number
is well-defined for any graph $G$.
\end{thm}

By Theorem \ref{T1} and noting that for a given graph $G$ every in-out-proper orientation  defines a proper vertex coloring for $G$, we have

\begin{equation}\label{E0}
 \lceil \frac{\chi(G)-1}{2} \rceil \leq  \overleftrightarrow{\chi}(G)\leq \Delta(G).
\end{equation}

\begin{example}
Let $G$ be a cycle. The degree of each vertex is two, so in each in-out-proper orientation of $G$, the in-out-degree of each vertex is $-2,+2$, or $0$. The graph $G$ has at least two adjacent vertices, so $\overleftrightarrow{\chi}(G) \geq 2$.
On the other hand, by Theorem \ref{T1}, $\overleftrightarrow{\chi}(G) \leq 2$.
Consequently, for every  cycle $C_n$ we have $\overleftrightarrow{\chi}(G)=2$.
\end{example}

Araujo {\it et al.} asked whether the proper orientation number of a planar graph is bounded.
We pose the following conjecture for the in-out-proper orientation number of planar graphs.

\begin{conj}\label{C1}
There is a constant number $c$ such that for every planar graph $G$, we have $ \overleftrightarrow{\chi}(G) \leq c$.
\end{conj}

Towards  Conjecture \ref{C1}, 
we study the in-out-proper orientation number of trees and show that for every tree $T$ we have $\overleftrightarrow{\chi}(T) \leq 3 $.

\begin{thm}\label{T2}
For every tree $T$ we have $\overleftrightarrow{\chi}(T) \leq 3 $ and this bound is sharp. 
\end{thm}

A graph is  called subcubic if it has maximum degree at most three. Let $G$ be a subcubic graph. By Theorem \ref{T1}, we have $ \overleftrightarrow{\chi}(G) \leq 3$. By using the properties of totally unimodular matrices
we show that there is a polynomial time algorithm to determine whether $\overleftrightarrow{\chi}(G) \leq 2$.

\begin{thm}\label{T3}
There is a polynomial time algorithm to determine whether $\overleftrightarrow{\chi}(G) \leq 2$, for a given graph $G$ with maximum degree three. 
\end{thm}
 
On the other hand, we show that it is NP-complete to decide whether $\overleftrightarrow{\chi}(G) \leq 1$ for a given bipartite graph $G$ with maximum degree three.
 
\begin{thm}\label{T4}
It is NP-complete to decide whether $\overleftrightarrow{\chi}(G) \leq 1$ for a given bipartite graph $G$ with maximum degree three.
\end{thm}

Next, we study the computational complexity of determining the the in-out-proper orientation number of  4-regular graphs. Note that for any 4-regular graph $G$ we have  $2 \leq \overleftrightarrow{\chi}(G) \leq 4$.

\begin{thm}\label{T5}
 It is NP-complete to decide whether $\overleftrightarrow{\chi}(G) \leq 2$ for a given 4-regular graph $G$.
\end{thm}

Let $G$ be a 4-regular graph with $\overleftrightarrow{\chi}(G) \leq 3$ and suppose that $D$ is an optimal  in-out-proper orientation. In $G$ the degree of each vertex is four, so the in-out-degree of each vertex is in $\{0,\pm 2\}$. Thus, we have $\overleftrightarrow{\chi}(G) \leq 3$ if and only if $\overleftrightarrow{\chi}(G) \leq 2$.   Thus, by Theorem \ref{T5}, we have the following corollary. 

\begin{cor}
 It is NP-complete to decide whether $\overleftrightarrow{\chi}(G) \leq 3$ for a given 4-regular graph $G$.
\end{cor}

The organization of the rest of the paper is as follows: In Section \ref{S2}, we present some definitions and notations. 
This is followed in Section \ref{S3} by 
some  bounds for the in-out-proper orientation number of graphs. Next, in Section \ref{S4}, we prove that the in-out-proper orientation number of each tree is at most three. In Section \ref{S5}, we focus on the in-out-proper orientation number of subcubic graphs. Section \ref{S6} is devoted to
the computational complexity of   regular graphs.
The paper is concluded with some remarks in Section \ref{S7}.

\section{Definitions}
\label{S2}

In this work, all graphs are finite and simple (i.e. without loops and multiple edges). We follow \cite{MR1367739} for terminology and notation where they are not defined here. If $G$ is a graph, then $V (G)$ and $E(G)$ denote the vertex set and the edge set of $G$, respectively. For every $v\in V(G)$, $d_G(v)$ denotes the degree of $v$ in the graph $G$. Also, $\Delta(G)$ denotes the maximum degree of $G$. The distance between two vertices $v$ and $u$, denoted by $distance(v,u)$, is the length of a shortest path between them.

An orientation $D$ of a graph $G$ is a digraph obtained from the graph $G$ by replacing each edge by just one of the two possible arcs with the same endvertices. For every vertex $v$,  the in-degree (the out-degree) of $v$ in the orientation $D$, denoted by
$d_D^{-}(v)$ ($d_D^{+}(v)$), is the number of arcs with head (tail) $v$ in $D$.  Also, the in-out-degree of $v$, denoted by $d_D^{\pm}(v)$, is defined as $d_D^{-}(v) - d_D^{+}(v)$.
An orientation of a graph $G$ is {\it in-out-proper} if any two adjacent vertices have different in-out-degrees.
The {\it in-out-proper orientation number} of a graph $G$, denoted by   $\overleftrightarrow{\chi}(G)$, is  $ \min_{D\in \Gamma}\max_{v\in V(G)} |d_D^{\pm}(v)|$, where $\Gamma$ is the set of in-out-proper orientations of $G$.

Let $G$ be a graph. A proper vertex $t$-coloring of  $G$ is a function $f: V(G) \longrightarrow \{1,\ldots,t\}$ such that if $u,v\in V(G)$ are adjacent, then $f(u)$ and $f(v)$ are different. The smallest integer $t$ such that $G$ has a proper vertex $t$-coloring is called the  chromatic number of $G$ and denoted by $\chi(G)$. 
Also, a proper edge $t$-coloring of  $G$ is a function $f: E(G) \longrightarrow \{1,\ldots,t\}$ such that if $e,e'\in E(G)$ have a same endvertex, then $f(e)$ and $f(e')$ are different. The smallest integer $t$ such that $G$ has a proper edge $t$-coloring is called the edge chromatic number (or chromatic index) of $G$ and denoted by $\chi'(G)$. 

For a graph $G = (V, E)$, the line graph of $G$   is a graph with the set of vertices $E(G)$ and two vertices  are adjacent if and only if their corresponding edges share a common endpoint in $G$.

A matrix $A$ is totally unimodular if every square submatrix of $A$ has determinant $1$, $0$ or $-1$. The importance of totally unimodular matrices stems from the fact that when an integer linear program has all-integer coefficients and the matrix of coefficients is totally unimodular, then the optimal solution of its relaxation is integral. Therefore, it can be obtained in polynomial time \cite{MR948455}.

\section{General bounds}
\label{S3}

For every graph $G$ we have $ \overleftrightarrow{\chi}(G) \leq \Delta(G)$. It is good to mention that the inequality is tight for any complete graph. For any $n$, each vertex of $K_n$ can have only in-out-degree $n-1,n-3, \ldots,-(n-3),-(n-1)$. The number of these values is exactly $n$. Thus, the in-out-proper orientation number of $K_n$ is at least $n-1=\Delta (K_n)$.

Next, we present some observation for the in-out-proper orientation number of graphs.

\begin{lem}
Let $G$ be a graph with at least one edge and assume that  $D$ is an in-out-proper orientation of $G$. Then in the orientation $D$ there is at least one vertex with positive in-out-degree and at least one vertex with negative in-out-degree.
\end{lem}

\begin{proof}
{
Let $G$ be a graph with at least one edge and assume that  $D$ is an in-out-proper orientation of $G$. First, we show that in $D$ there is a vertex with positive in-out-degree. To the contrary assume that the in-out-degree of each vertex is negative or zero. So, we have
\begin{equation}\label{EE1}
 \sum_{v\in V(G)} d_D^{\pm}(v) \leq 0.   
\end{equation}
On the other hand, we have
\begin{equation}\label{EE2}
 \sum_{v\in V(G)} d_D^{-}(v) = \sum_{v\in V(G)} d_D^{+}(v). 
\end{equation}
Thus, by (\ref{EE1}) and (\ref{EE2}), we conclude that for every vertex $v$ we have $d_D^{\pm}(v)=0$. The graph $G$ has at least one edge, but in $D$ the in-out-degrees of all vertices are zero (so, it is not a proper vertex coloring). Thus $D$ is not an in-out-proper orientation for $G$. This is a contradiction. So,  there is a vertex with positive in-out-degree. Similarly,  we can show that there is a vertex with negative in-out-degree.
}
\end{proof}

\section{Trees}
\label{S4}

Next,  we study the in-out-proper orientation number of trees and show that for every tree $T$ we have $\overleftrightarrow{\chi}(T) \leq 3 $.
Also, we show that this bound is sharp.

\begin{proof}[Proof of Theorem~\ref{T2}]
First we show that for each tree $T$ we have $\overleftrightarrow{\chi}(T) \leq 3 $. Let $T$ be a tree with $n$ vertices and $v$ be a vertex of $T$.
Sort the vertices of $T$ according to their distance from $v$ and let $v=v_1,v_2,\ldots,v_n$ be  that sorted set. For each vertex $u$, 
the father of $u$, denoted by $f(u)$, is the unique vertex that is adjacent  and closer to the root $v$.
Perform the Algorithm \ref{A2} and call the resultant orientation $D$.

\begin{algorithm}
\caption{}\label{A2}
\small
\begin{algorithmic}[1]
\For{$i=1$ to $n$}
    \If{$i=1$}
       \State{Orient the edges incident with $v_1$ such that if $d(v_1)$ is an even number then $d_D^{\pm}(v_1)=2$, and if $d(v_1)$ is an odd number then $d_D^{\pm}(v_1)=1$.}
    \ElsIf{$distance(v_i,v_1)$ is an even number}
       \If{the edge $v_i f(v_i)$ was oriented from $f(v_i)$ to $v_i$}
           \State{Orient the set of edges $\{v_iv_j | j> i\}$
           such that if $d(v_i)$ is an odd number then $d_D^{\pm}(v_i)=1$, and if $d(v_i)$ is an even number then $d_D^{\pm}(v_i)=2$.}
        \ElsIf{the edge $v_i f(v_i)$ was oriented from $v_i$ to $f(v_i)$}
           \If{$d(v_i) \geq 3$}
             \State{Orient the set of edges $\{v_iv_j | j> i\}$
             such that  $d_D^{\pm}(v_i)\in \{1,2\}$}
           \ElsIf{ $d(v_i)=2$}
              \If{$d_D^{\pm}(f(v_i))\neq 0 $}
                \State{Orient the  edge $\{v_iv_j | j> i\}$
             such that  $d_D^{\pm}(v_i)=0$}
              \ElsIf{$d_D^{\pm}(f(v_i))=0 $}
                \State{Orient the set of edges incident with $v_i$ such that
                $d_D^{\pm}(v_i)=2$ (note that we reorient the edge $v_i f(v_i)$.}
              \EndIf
           \ElsIf{ $d(v_i)=1$}
              \If{$d_D^{\pm}(f(v_i))= -1$}
                \State{Reorient the edge $v_if(v_i)$ from $f(v_i)$ to $v_i$}
              \EndIf
           \EndIf
        \EndIf
    \ElsIf{$distance(v_i,v_1)$ is an odd number}
    %%%%%%%%%%%%%%%%%%%%
       \If{the edge $v_i f(v_i)$ was oriented from $v_i$ to $f(v_i)$}
           \State{Orient the set of edges $\{v_iv_j | j> i\}$
           such that if $d(v_i)$ is an odd number then $d_D^{\pm}(v_i)=-1$, and if $d(v_i)$ is an even number then $d_D^{\pm}(v_i)=-2$.}
        \ElsIf{the edge $v_i f(v_i)$ was oriented from $f(v_i)$ to $v_i$}
           \If{$d(v_i) \geq 3$}
             \State{Orient the set of edges $\{v_iv_j | j> i\}$
             such that  $d_D^{\pm}(v_i)\in \{-1,-2\}$}
           \ElsIf{ $d(v_i)=2$}
              \If{$d_D^{\pm}(f(v_i))\neq 0 $}
                \State{Orient the  edge $\{v_iv_j | j> i\}$
             such that  $d_D^{\pm}(v_i)=0$}
              \ElsIf{$d_D^{\pm}(f(v_i))=0 $}
                \State{Orient the set of edges incident with $v_i$ such that
                $d_D^{\pm}(v_i)=-2$ (note that we reorient the edge $v_i f(v_i)$.}
              \EndIf
           \ElsIf{$d(v_i)=1$}
              \If{$d_D^{\pm}(f(v_i))= 1$}
                \State{Reorient the edge $v_if(v_i)$ from $v_i$ to $f(v_i)$}
              \EndIf
           \EndIf
        \EndIf
    \EndIf
\EndFor
\end{algorithmic}
\end{algorithm}

We have the following properties for the orientation $D$ that we obtained from Algorithm \ref{A2}.

\begin{prop} \label{L1}
Let $u$ be a vertex with $d(u)\geq 3 $. If 
$u$ has an even distance from the root $v_1$, then  
 $d_D^{\pm}(u)\in \{1,2,3\}$. Also, if $u$ has an odd distance from the root $v_1$, then  
 $d_D^{\pm}(u)\in \{-1,-2,-3\}$.   
\end{prop}

\begin{subproof}{
Let $u$ be a vertex with $d(u)\geq 3 $. If $u$ has an even distance from the root $v_1$, then 
at  Lines 2-3, 5-6, and 8-9, we orient the edges incident with $u$ such that the in-out-degree of $u$ is in $\{1,2,3\}$. There is only one other part of the algorithm that we may change the in-out-degree of $u$.
That part is Lines 35-36. In that case the in-out-degree of $u$ is one and by reorienting one of the edges that is incident with $u$ we  increase the  in-out-degree of $u$ by two. Similarly, if $u$ has an odd distance from the root $v$, then 
at  Lines  23-24, 26-27, and 17-18,  we orient the edges incident with $u$ such that the in-out-degree of $u$ is in $\{-1,-2,-3\}$.
}\end{subproof}

\begin{prop}\label{L2}
Let $u$ be a vertex with $d(u)=2 $. If 
$u$ has an even distance from the root $v_1$, then  
 $d_D^{\pm}(u)\in \{0,1,2\}$. Also, if $u$ has an odd distance from the root $v_1$, then  
 $d_D^{\pm}(u)\in \{0,-1,-2\}$.   
\end{prop}

\begin{subproof}{
Let $u$ be a vertex with $d(u)=2 $. If $u$ has an even distance from the root $v_1$, then 
at  Lines 2-3, 5-6, and 10-14, we orient the edges incident with $u$ such that the in-out-degree of $u$ is in $\{0,1,2\}$. There is only one other part of the algorithm that we may change the in-out-degree of $u$.
That part is Lines 31-32. In that case the in-out-degree of $u$ is zero and by reorienting one of the edges that is incident with $u$ we  increase the  in-out-degree of $u$ by two. So, the final in-out-degree of $u$ is in $\{0,1,2\}$. Similarly, if $u$ has an odd distance from the root $v_1$, then 
at  Lines  23-24, 28-32, and 13-14,  we orient the edges incident with $u$ such that the in-out-degree of $u$ is in $\{0,-1,-2\}$.
}
\end{subproof}

\begin{prop}\label{L3}
Let $u$ and $u'$ be two adjacent vertices such that $d(u)=d(u')=2$. Then  $d_D^{\pm}(u) \neq d_D^{\pm}(u')$.
\end{prop}

\begin{subproof}{
By Lines 10-14 and Lines 28-32, the algorithm does not produce any two adjacent vertices $u,u'$ such that $d(u)=d(u')=2$ and $d_D^{\pm}(u) = d_D^{\pm}(u')=0$. Thus, by Proposition \ref{L2}, for any two adjacent vertices $u,u'$ with $d(u)=d(u')=2$ we have $d_D^{\pm}(u) \neq d_D^{\pm}(u')$.
}
\end{subproof}

\begin{prop}\label{L4}
Let $u$ be a vertex with $d(u)=1 $. 
Then $ d_D^{\pm}(u)\in \{-1,+1\}$ and $ d_D^{\pm}(u)\neq d_D^{\pm}(f(u))$.
\end{prop}

\begin{subproof}{
Let $u$ be a vertex with $d(u)=1 $. If $u$ has an even distance from the root $v_1$, then 
at  Lines 2-3, 5-6 and 16-19, we orient the edge incident with $u$ such that the in-out-degree of $u$ is in $\{-1,+1\}$ and also  if it is $-1$ then $ d_D^{\pm}(u)\neq d_D^{\pm}(f(u))$. By Propositions \ref{L1}, \ref{L2}, we also conclude that  if the in-out-degree of $u$ is $1$, then $ d_D^{\pm}(u)\neq d_D^{\pm}(f(u))$.
Similarly, if $u$ has an odd distance from the root $v_1$, then 
at  Lines  23-24, and 34-37,  we orient the edge incident with $u$ such that the in-out-degree of $u$ is in $\{-1,+1\}$ and if it is $1$, then $ d_D^{\pm}(u)\neq d_D^{\pm}(f(u))$. By Propositions \ref{L1}, \ref{L2}, we  conclude that  if the in-out-degree of $u$ is $-1$, then $ d_D^{\pm}(u)\neq d_D^{\pm}(f(u))$.
This completes the proof.
}\end{subproof}

By Propositions \ref{L1}, \ref{L2}, \ref{L3} and \ref{L4}, for every vertex $u$, we have $ d_D^{\pm}(u)\in \{\pm 3, \pm 2, \pm 1 ,0\}$ and for every two   adjacent vertices $u,u'$, we have $d_D^{\pm}(u) \neq d_D^{\pm}(u')$. Thus $D$ is an in-out-proper orientation such that the maximum  of absolute values of their in-out-degrees is at most three. 

Finally, we show that there is a tree $T$ such that $\overleftrightarrow{\chi}(T) = 3  $. Consider the tree $T$ with the set of vertices $v_1,v_2,v_3,v_4$ and set of edges $v_1v_2,v_1v_3, v_1v_4$. We have $d(v_2)=d(v_3)=d(v_4)=1$, so in any orientation of $T$, their in-out-degrees are in $\{\pm 1\}$. On the other hand, the degree of $v_1$ is three, so its in-out-degree is in $\{\pm 1, \pm 3\}$. To the contrary assume that $\overleftrightarrow{\chi}(T) < 3  $, and let $D$ be an in-out-proper orientation of $T$ such that $d_D^{\pm }(v_1)\in \{\pm 1\}$. The in-out-degree of at least one of the vertices $v_2,v_3,v_4$ is $1$ (otherwise $d_D^{\pm }(v_1)\notin \{\pm 1\}$)  and also the in-out-degree of at least one of the vertices $v_2,v_3,v_4$ is $-1$. So, $D$ has two adjacent vertices with the same in-out-degree Thus, it is not an in-out-proper orientation. This is a contradiction. So, we conclude that $\overleftrightarrow{\chi}(T) = 3  $.
\end{proof}

\section{Subcubic graphs}
\label{S5}

In this section we focus on subcubic graphs. 
Let $G$ be a subcubic graph. By Theorem \ref{T1}, we have $ \overleftrightarrow{\chi}(G) \leq 3$. Next,
we show that there is a polynomial time algorithm to determine whether $\overleftrightarrow{\chi}(G) \leq 2$. On the other hand, it is NP-complete to decide whether $\overleftrightarrow{\chi}(G) \leq 1$ for a given bipartite graph $G$ with maximum degree three.

\begin{proof}[Proof of Theorem~\ref{T3}]
Let $G$ be a cubic graph. If $D$ is an optimal in-out-proper orientation, then for any vertex $v$ of degree two we have $d_D^{\pm}(v)\in \{0,\pm 2\}$ and also for any vertex $v$ of degree  one or three we have $d_D^{\pm}(v)\in \{\pm 1\}$. 
First, we investigate the subcubic graphs without degree two vertices. Then, we present a polynomial time algorithm for subcubic graphs.
Let $G$ be a graph such that the degree of each vertex is one or three and without loss of generality assume that $G$ is connected. Also, suppose that $\overleftrightarrow{\chi}(G) \leq 2$ and $D$ is an optimal in-out-proper orientation of $G$.
Since $d_D^{\pm}(v)\in \{\pm 1\}$ for each vertex $v$ and the in-out-degrees form a proper vertex coloring of $G$, $G$ should be bipartite.

\begin{prop}\label{P1}
Let $G$ be a graph such that the degree of each vertex is one or three. If 
$\overleftrightarrow{\chi}(G) \leq 2$, then $G$ is bipartite.
\end{prop}

Consequently, at step one we should check that whether $G$ is bipartite. Next, at step two we want to determine whether it is possible to orient the edges of $G$ such that in-out-degrees of vertices of one partite set of $G$ are $1$ and  in-out-degrees of  vertices of the other partite set of $G$ are $-1$.

\begin{prop}\label{P2}
Let $G=(X\cup Y,E)$ be a bipartite graph such that the degree of each vertex is one or three. If 
$\overleftrightarrow{\chi}(G) \leq 2$, then there is an orientation for the edges of $G$ such that the in-out-degree of each vertex is $1$ or $-1$, and the in-out-degrees of all vertices in $X$ are the same, and so are those in $Y$.
\end{prop}

It is well-known that there is a polynomial time algorithm to decide whether a given graph is bipartite \cite{MR1367739}. Next, we present a polynomial time algorithm for step two. Let $G=(X\cup Y,E(G))$ be a bipartite graph such that the degree of each vertex is one or three. Without loss of generality assume that $X={x_1,x_2,\ldots,x_{n}}$ and $Y={y_1,y_2,\ldots,y_{n'}}$. From the graph $G$ we construct a bipartite graph $H$ with vertex set $V (H) = (U_X\cup  U_Y )\cup U_0$ and edge set $E(H) = W$. Put
$U_{X}=X$, and 
$U_{Y}=Y$. Also, for every edge $x_i y_j\in E(G)$, put $x_{i,j} $ in $U_0$ and the edges $w_{i,j}=x_ix_{i,j},w_{i,j}'=x_{i,j}y_j $ in $W$.
See Fig. \ref{PP01}.

\begin{figure}[]
\centering
\includegraphics [width=0.5\textwidth]{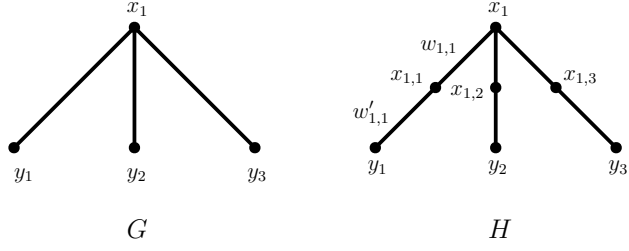}
\caption{The graph $G$ and its corresponding graph $H$.}
\label{PP01}
\end{figure}

Consider the following integer linear   program for the graph $H$.

\begin{alignat}{7}
& \text{Maximize}  &\quad& 1                           &\quad&                                        &\nonumber \\
& \text{subject to}& &\sum_{x_i y_j\in E(G) } w_{i,j}=1&     & \forall x_i\in U_X\text{ s.t. } d_G(x_i)=1& \label{L6}\\
&                  & &\sum_{x_i y_j\in E(G) } w_{i,j}=2& & \forall x_i\in U_X\text{ s.t. } d_G(x_i)=3& \label{L7}\\
&                  & &\sum_{x_i y_j\in E(G) } w_{i,j}'=0& & \forall y_j\in U_Y\text{ s.t. } d_G(y_j)=1&  \label{L8}\\
&                  & &\sum_{x_i y_j\in E(G) } w_{i,j}'=1& & \forall y_j\in U_Y\text{ s.t. } d_G(y_j)=3& \label{L9}\\
&                  & & w_{i,j}+w_{i,j}'=1               & & \forall x_{i,j}\in U_0                     & \label{L10}\\
&                  & & w_{i,j},w_{i,j}'\in \{0,1\}      & & \forall x_i y_j\in E(G)                     & \label{L11}
\end{alignat}
\\
Note that we can write the above integer linear  program in the following canonical form:

\begin{alignat}{3}
& \text{Maximize}  &\quad& 1               &\nonumber \\
& \text{subject to}&     &A{\bf x}={\bf b} & \label{LL0}\\
&                  &.    &{\bf x}\in \{0,1\}^{|E(H)|}, &  \label{LLL}
\end{alignat}
\\
where  $A$ is the incidence matrix of $H$, ${\bf x}^T=(w_{i_1,j_1}, \ldots, w_{i_k,j_k}')$, and ${\bf b}\in \{0,1,2\}^{|V(H)|}$. For instance, for the graph $H$ that was shown in Fig. \ref{PP01}, we have ${\bf x}^T=(w_{1,1}, w_{1,1}',w_{1,2}, w_{1,2}',w_{1,3}, w_{1,3}')$, ${\bf b}=(2,0,0,0,1,1,1)$ and
$A$ is
\[
\begin{blockarray}{ccccccc}
&w_{1,1} & w_{1,1}' & w_{1,2} & w_{1,2}' & w_{1,3} & w_{1,3}' \\
\begin{block}{c(cccccc)}
  x_1     & 1 & 0 & 1 & 0 & 1 & 0 \\
  y_1     & 0 & 1 & 0 & 0 & 0 & 0 \\
  y_2     & 0 & 0 & 0 & 1 & 0 & 0 \\
  y_3     & 0 & 0 & 0 & 0 & 0 & 1 \\
  x_{1,1} & 1 & 1 & 0 & 0 & 0 & 0 \\
  x_{1,2} & 0 & 0 & 1 & 1 & 0 & 0 \\
  x_{1,3} & 0 & 0 & 0 & 0 & 1 & 1 \\
\end{block}
\end{blockarray}
 \]

For each edge $x_i y_j\in E(G)  $ in (\ref{L11}), we consider two variables $w_{i,j},w_{i,j}'$ such that $ w_{i,j},w_{i,j}'\in \{0,1\}$. On the other hand, in (\ref{L10}), we have $w_{i,j}+w_{i,j}'=1$, so the value of exactly one of these variables is one and the value of the other variable  is zero. We consider the values of  $w_{i,j},w_{i,j}'$ as an orientation   for the edge $x_i y_j$ in $G$ such that it is oriented from $y_j$ to $x_i$ if and only if $w_{i,j}=1$. So, the values of the variables correspond to an orientation for the graph $G$. Call that orientation $D$. By (\ref{L6}) and (\ref{L7}), we ensure that in $D$ the in-out-degree of each vertex in $X$ is $1$. Also, by (\ref{L8}) and (\ref{L9}), the in-out-degree of each vertex in $Y$ is $-1$. Consequently, the above integer linear program is feasible if and only if the graph $G$  has an in-out-proper orientation such that the in-out-degree of each vertex in $X$ is $1$ and the in-out-degree of each vertex in $Y$ is $-1$.

When an integer linear program has all-integer coefficients and the matrix of coefficients is totally unimodular, then the optimal solution of its relaxation is integral. Therefore, it can be obtained in polynomial time \cite{MR948455}. On the other hand, it is shown in
[\cite{MR948455}, Corollary 2.9 in Page 544]
that every incidence matrix of a bipartite graph is totally unimodular. So, in our integer linear program the matrix of coefficients is totally unimodular. 
Consequently, there is a polynomial time algorithm to determine whether the above-mentioned integer linear program is feasible. 

Note that there is  an orientation of the edges of $G$ such that  the in-out-degree of each vertex in $X$ is $-1$ and the in-out-degree of each vertex in $Y$ is $1$ if and only if there is an orientation of the edges of $G$ such that  the in-out-degree of each vertex in $X$ is $1$ and the in-out-degree of each vertex in $Y$ is $-1$ (by considering the reverse of the given orientation). 
Consequently, there is  a polynomial time algorithm to determine whether the in-out-proper orientation number of given graph $G$ with degree set $\{1,3\}$ is at most two. 

Next, we consider the set of subcubic graphs. Let $G$ be a subcubic graph. If $D$ is an optimal in-out-proper orientation, then for any vertex $v$ of degree two we have $d_D^{\pm}(v)\in \{0,\pm 2\}$ and also for any vertex $v$ of degree  one or three we have $d_D^{\pm}(v)\in \{\pm 1\}$. 

Remove all vertices of degree two from the graph $G$ and call the resultant graph $G'$. For each vertex $v$ in $G'$ if $d_G(v) \neq d_{G'}(v)$, then put $d_G(v) - d_{G'}(v)$ isolated vertices and join them to $v$ (we call these new vertices  dummy vertices). Call the resultant graph $G''$. Note that in $G''$ the degree of each vertex is one or three. See Fig. \ref{PP02}.

\begin{figure}[]
\centering
\includegraphics [width=0.7\textwidth]{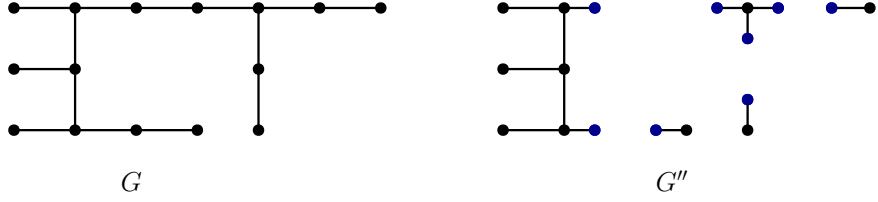}
\caption{The graph $G$ and its corresponding graph $G''$. In the graph $G''$ the degree of each vertex is 1 or 3 and the set of blue vertices are dummy vertices.}
\label{PP02}
\end{figure}

Assume that $\overleftrightarrow{\chi}(G) \leq 2$. next, we present some necessary conditions for $G''$.

\begin{prop}\label{PR1}
Let $C_1,C_2,\ldots,C_k$ be all the connected components of $G''$. For any $i\in \{1,2,\ldots,k\}$, ($C_i$ is
bipartite and) there exists an orientation $D_i$ of $C_i$ satisfying\\
(a) every vertex of $C_i$ has the in-out-degree $1$ or $-1$, and\\
(b) for any $uv\in E(G) \cap E(G'')$, $d_{D_i}^{\pm}(u)\neq d_{D_i}^{\pm}(v)$.

\end{prop}

\begin{subproof}
By Proposition \ref{P1} and Proposition \ref{P2} the proof is clear.
\end{subproof}

In order to complete the proof we do the following steps:\\
Step 1. Proving that the condition in Proposition \ref{PR1} is a necessary and sufficient one for $\overleftrightarrow{\chi}(G) \leq 2$. \\
Step 2. Showing that the condition in Proposition \ref{PR1} can be checked in polynomial time.\\
Step 3. Concluding that Theorem \ref{T3} is true by Step 1 and Step 2.

(Proof of Step 1:) We show that the necessary conditions that are presented in Proposition \ref{PR1} are also sufficient. In other words, we prove that if each connected component of $G''$ is bipartite and also
if the graph $G''$ has an orientation such that in each connected component,  the in-out-degrees of vertices in different parts (of that bipartite component), except  dummy vertices, are different, 
then we can extend that partial orientation to an in-out-proper orientation of $G$ such that the maximum of absolute values of their in-out-degrees is at most two. 
To prove that it is enough, we show the following proposition.

\begin{prop}\label{P3}
Each path  $P_n=v_1,v_2,\ldots,v_n$ of length at least two (i.e. $n\geq 3$) has the following four kinds of in-out-proper orientations:\\
(1) The in-out-proper orientation $D_1$ such that $d_D^{\pm}(v_1)=d_D^{\pm}(v_n)=1$. \\
(2) The in-out-proper orientation $D_2$ such that $d_D^{\pm}(v_1)=d_D^{\pm}(v_n)=-1$. \\
(3) The in-out-proper orientation $D_3$ such that $d_D^{\pm}(v_1)=1$ and $d_D^{\pm}(v_n)=-1$. \\
(4) The in-out-proper orientation $D_4$ such that $d_D^{\pm}(v_1)=-1$ and $d_D^{\pm}(v_n)=1$. 
\end{prop}

\begin{subproof}
{
Let $P_n=v_1,v_2,\ldots,v_n$ be a path of length at least two and $D\in \{D_1,D_2,D_3,D_4\}$. 
Orient the edges $v_1v_2$ and $v_{n-1}v_n$ such that the in-out-degree of $v_1$ is  $d_D^{\pm}(v_1)$ and the in-out-degree of $v_n$ is  $d_D^{\pm}(v_n)$. Do Algorithm \ref{A3} to orient the remaining edges. 

\begin{algorithm}
\caption{}\label{A3}
\small
\begin{algorithmic}[1]
\For{$i=2$ to $n-2$}
   \If{$d_{D}^{\pm}(v_{i-1})=0$  and  $v_{i-1}v_i$ was oriented from  $v_{i-1}$  to $v_i$}
       \State{Orient  $v_i v_{i+1}$ from $v_{i+1}$ to $v_{i}$}
   \Else
       \State{Orient  $v_i v_{i+1}$ from $v_{i}$ to $v_{i+1}$ }
   \EndIf
\EndFor
\end{algorithmic}
\end{algorithm}

By Algorithm \ref{A3}, there is no two consecutive vertices with the in-out-degree $0$. On the other hand, it not possible to have two consecutive vertices with the in-out-degree $2$ or $-2$. Moreover, the in-out-degree of $v_n$ is in $\{\pm 1\}$ and the in-out-degree of $v_{n-1}$ is in $\{0, \pm 2\}$. So $D$ is an in-out-proper orientation. 
}
\end{subproof}

(Proof of Step 2:) To check whether $G''$ has such an orientation we can use the previous mentioned integer linear program with some modifications. In fact, for each dummy vertex $v$ we just remove the corresponding condition in (\ref{L6}) or (\ref{L8}), and then we solve the integer linear program.

(Proof of Step 3:) Having Propositions \ref{P2}, and \ref{P3}, and noting that there is a polynomial time algorithm to check  Proposition \ref{P2}, we conclude that there is a polynomial time algorithm to decide whether in-out-proper orientation number of a given subcubic graph is at most two.   
This completes the proof.
\end{proof}

Next, we prove that it is NP-complete to decide whether $\overleftrightarrow{\chi}(G) \leq 1$ for a given bipartite graph $G$ with maximum degree three. 

\begin{proof}[Proof of Theorem~\ref{T4}]
It was shown in \cite{MR3386014} that the following variant of Not-All-Equal  satisfying assignment problem
is NP-complete.

\smallskip

\noindent \textbf{Problem}: {\sc Cubic Monotone Not-All-Equal (2,3)-Sat}.\\
\textsc{Input}: Set $X$ of variables, collection $C$ of clauses over $X$ such that every clause $c\in C$ has $| c |\in \{2, 3\}$, each variable appears in exactly three clauses and there is no negation in the formula.\\
\textsc{Question}: Is there a truth assignment for $X$ such that every clause in $C$ has at least one true literal and at least one false literal?

\smallskip

Our proof is a polynomial time reduction from {\sc Cubic Monotone Not-All-Equal (2,3)-Sat}. Let $\Phi $ be an instance with the set of  variables
$X $ and the set of clauses $C$. We transform it to a bipartite graph $G_{\Phi}$ with maximum degree three in polynomial time such that $\overleftrightarrow{\chi}(G_{\Phi}) \leq 1$ if and only if $\Phi $ has a  Not-All-Equal truth assignment.
We use the auxiliary gadget $I_x$ which is shown in Fig. \ref{PP03}. Our construction consists of three steps.\\
{\bf Step 1.} For each variable $x\in X$ put a copy of the gadget $I_x$ which is shown in Fig. \ref{PP03}. \\
{\bf Step 2.} For each clause $c\in C$ put a vertex $c$ and then for each variable $x$ that appears in the clause $c$ join the vertex $c$ to one of the vertices $x_1,x_2,x_3$ of $I_x$ such that in the resultant graph for each variable $x\in X$ in the gadget $I_x$ the degrees of the variables $x_1,x_2,x_3$ are two. Call the resultant graph $H_{\Phi}$.\\
{\bf Step 3.} For each clause $c=(x \vee x')\in C$, without loss of generality assume that  $cx_1,cx_1'\in E(H_{\Phi})$. Merge the three vertices $c,x_1,x_1'$ into a new vertex $c'$.

\begin{figure}[]
\centering
\includegraphics [width=0.7\textwidth]{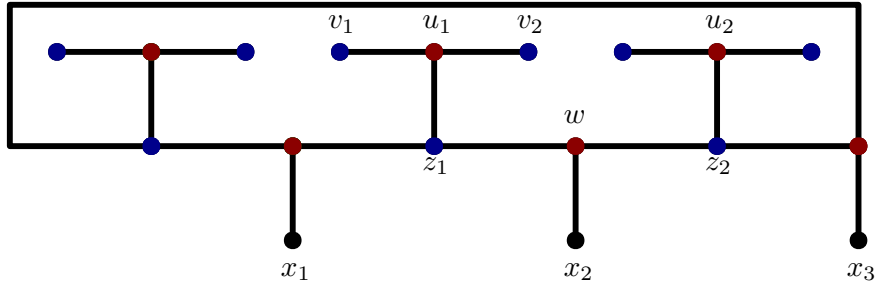}
\caption{The gadget $I_x$.}
\label{PP03}
\end{figure}

Call the resultant graph $G_{\Phi}$. The degree of every vertex in the graph $G_{\Phi}$ is 1,  2 or 3 and the resultant graph is bipartite. Let us now prove that $\overleftrightarrow{\chi}(G_{\Phi}) \leq 1$ if and only if $\Phi$ has a Not-All-Equal truth assignment. 

First, assume that $\overleftrightarrow{\chi}(G_{\Phi}) \leq 1$. We have the following properties. 

\begin{prop}\label{P8}
Consider the gadget $I_x$ which is shown in Fig. \ref{PP03}. Let $D$ be an orientation of $I_x$ such that the in-out-degree of each vertex is in $\{0,\pm 1\}$ and  the endvertices of any edge in $I_x$, except three edges incident with the vertices $x_1,x_2,x_3$,  have different in-out-degrees, then $d_D^{\pm}(x_1)=d_D^{\pm}(x_2)= d_D^{\pm}(x_3)= 1 $ or
$d_D^{\pm}(x_1)=d_D^{\pm}(x_2)= d_D^{\pm}(x_3)=-1 $.
\end{prop}

\begin{subproof}
{
Note that  in the proof of this proposition the notation and colors that we refer 
 are depicted in 
Fig. \ref{PP03}.
In the orientation $D$ the in-out-degree of each vertex is in $\{0,\pm 1\}$. On the other hand, the degree of each vertex is one or three, so  the in-out-degree of each vertex is in $\{\pm 1\}$. In orientation $D$ the endvertices of any edge, except three edges incident with the vertices $x_1,x_2,x_3$,  have different in-out-degrees. Thus, the  red vertices have the same in-out-degree and also the   blue vertices have the same in-out-degree. Now, two cases can be considered.\\
{\bf Case 1.} The blue vertices have the in-out-degree $1$. Then the red vertices have the in-out-degree $-1$. We have $d_D^{\pm}(v_1)= d_D^{\pm}(v_2)= 1$, so the edges $v_1u_1, v_2u_1$ were oriented form $u_1$ to $v_1$ and $v_2$, respectively. The   in-out-degree of $u_1$ is $-1$. Thus, the edge $z_1u_1$ was oriented from $z_1$ to $u_1$. The in-out-degree of $z_1$ is $1$ and thus the edge $wz_1$was oriented from $w$ to $z_1$. 
We have the same situation for $z_2$. Its in-out-degree is $1$ and  the edge $wz_2$ was oriented from $w$ to $z_2$.  The vertex $w$ is a red vertex and its in-out-degree is $-1$. On the other hand, the edges $wz_1,wz_2$ were oriented from $w$ to $z_1$ and $z_2$. Thus, the edge $wx_2$ was oriented from $w$ to $x_2$ and consequently $d_D^{\pm}(x_2)=1$. We have the same conclusion for $x_1$ and $x_3$.
Thus, $d_D^{\pm}(x_1)=d_D^{\pm}(x_2)= d_D^{\pm}(x_3)= 1 $. \\
{\bf Case 2.} The blue vertices have the in-out-degree $-1$. Then the red vertices have the in-out-degree $1$. Similar to Case 1, we can show that $d_D^{\pm}(x_1)=d_D^{\pm}(x_2)= d_D^{\pm}(x_3)= -1 $. This completes the proof.
}
\end{subproof}

Let $D$ be an optimal in-out-proper orientation of $G_{\Phi}$.
Now, we present a Not-All-Equal truth assignment for the formula $\Phi$. Let $\Gamma : X \rightarrow \{ {\sf true},{\sf false} \} $ be the assignment defined by $\Gamma( x_i)= {\sf true}$  if the blue vertices in $I_x$ have the in-out-degree $1$, and $\Gamma( x_i)= {\sf false}$ if the blue vertices in $I_x$ have the in-out-degree $-1$.

\begin{figure}[]
\centering
\includegraphics [width=0.7\textwidth]{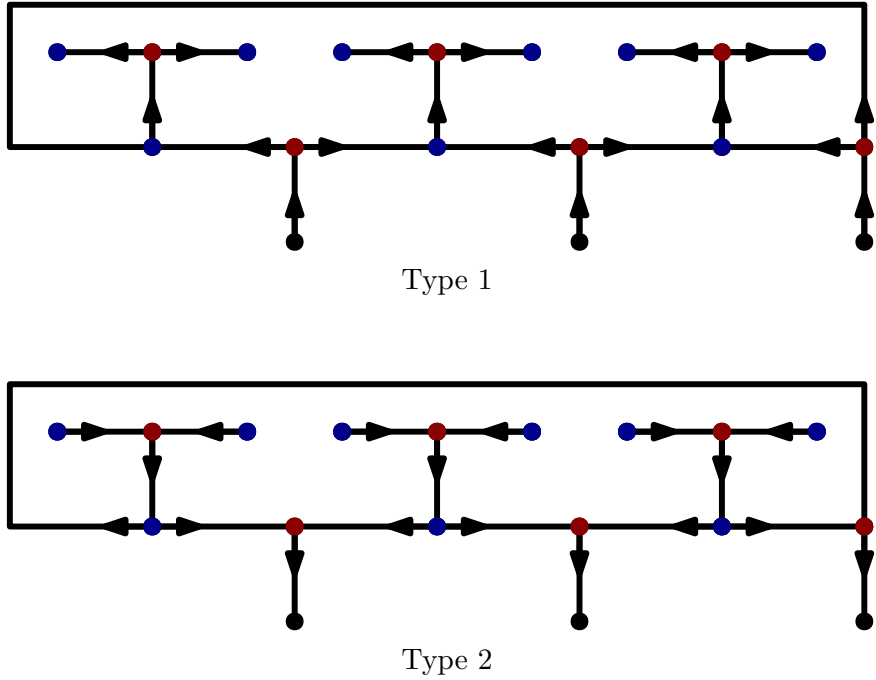}
\caption{The two possible orientations of $I_x$.}
\label{PP04}
\end{figure}

Next, we prove that $\Gamma$  is a  Not-All-Equal truth assignment for $\Phi$. Let $c=(x\vee y \vee r)$ and without loss of generality assume that $c x_1,c y_1,c r_1\in E(G_{\Phi})$.  The degree of the vertex $c$ is three, so 
$d_D^{\pm}(c)\in \{\pm 1\}$. Thus, at least one of the edges incident with $c$ was oriented from $c$ to the other endpoint. Note that the other endpoint is one of the vertices $x_1,y_1,r_1$. Also, at least one of the edges incident with $c$ was oriented toward $c$. 
On the other hand, the degree of vertices $x_1,y_1,r_1$ are two, so $d_D^{\pm}(x_1)=d_D^{\pm}(y_1)= d_D^{\pm}(r_1)= 0 $. Thus, ${\sf true},{\sf false} \in \{\Gamma (x), \Gamma (y), \Gamma (r)\}$.
Next, assume that $c=(x\vee y)$.  The degree of the vertex $c'$ (that corresponds to the clause $c$ in $C$) is two. So, $d_D^{\pm}(c')= 0$. Thus,  ${\sf true},{\sf false} \in \{\Gamma (x), \Gamma (y) \}$.

Now, assume that there is a Not-All-Equal assignment $\Gamma : X \rightarrow \{ {\sf true},{\sf false} \} $ for $\Phi$. For each variable $x\in X$ if $\Gamma(x)={\sf true}$ then orient $I_x$ like Type 2 in Fig. \ref{PP04} and if $\Gamma(x)={\sf false}$ then orient $I_x$ like Type 1 in Fig. \ref{PP04}. Also, for each clause $c=(x\vee y\vee r )$ orient the edges incident with $c$ such that the in-out-degree of each neighbor of $c$ is $0$. Call the resultant orientation $D$. The function $\Gamma $ is a  Not-All-Equal assignment, so $D$ is an in-out-proper orientation such that the maximum of absolute values of  their in-out-degrees is one. This completes the proof.  
 
\end{proof}

\section{Regular graphs}
\label{S6}

Next, we study the computational complexity of determining the  in-out-proper orientation number of  4-regular graphs.

\begin{proof}[Proof of Theorem~\ref{T5}]
It was shown that it is NP-complete to determine whether the edge chromatic number of a given 3-regular graph is three (see \cite{MR1367739}). We reduce this problem to our problem in polynomial time. For a given 3-regular graph $G$ we construct a 4-regular graph $H$ such that the edge chromatic number of $G$ is three if and only if 
 $\overleftrightarrow{\chi}(H) \leq 2$.

For a given graph $G$ with the set of edges $e_1,e_2,\ldots, e_n$,  let $H$ be the line graph of $G$ with the set of vertices $v_{e_1},v_{e_2},\ldots, v_{e_n}$, such that $v_{e_i}v_{e_j}\in E(H)$ if and only if $e_i$ and $e_j$ have a common endvertex.  First, assume that
the in-out-proper orientation number of $H$ is two and let $D$ be an optimal  in-out-proper orientation. The orientation $D$ defines a proper vertex 3-coloring for the vertices of $H$ using three colors $0,\pm 2$. Thus, $G$ has a proper edge 3-coloring. 

Next, assume that the edge chromatic number of $G$ is three and let $f:E(G) \rightarrow \{1,2,3\}$ be a proper edge 3-coloring of $G$. Define the function $h:V(H)\rightarrow \{1,2,3\}$ such that $h(e_{v_i})=k$ if and only if $f(e_i)=k$, for each $k=1,2,3$. Let $K$ be the subset of edges of $H$ such that for each edge $v_{e_i}v_{e_j}\in K$ we have $\{ h(v_{e_i}), h(v_{e_j}) \}=\{1,3\}$. In the subgraph $H \setminus K$ the degree of each vertex is even. In fact the degree of each vertex $v_{e_i} $ with $h(v_{e_i})=2$ is four and also
the degree of each vertex $v_{e_i} $ with $h(v_{e_i})\in \{1,3\}$ is two.
So we can orient the edges in $H \setminus K$ such that the in-degree of each vertex is equal to its out-degree. Next, for each edge $v_{e_i}v_{e_j}\in K$ orient it from $v_{e_i}$ to $v_{e_j}$ if $h(v_{e_i})=1$ and $h(v_{e_j})=3$, otherwise orient it from $v_{e_j}$ to $v_{e_i}$. Consider the union of orientations for $H \setminus K$ and $K$ and call the resultant orientation $D$. In $D$ the in-out-degree of each vertex $v_{e_i}$ with $ h(v_{e_i})=1$ ($ h(v_{e_i})=2$, $ h(v_{e_i})=3$, respectively) is $-2$ ($0,2$, respectively). Thus, $D$ is an in-out-proper orientation such that the maximum of absolute values of their in-out-degree is two. This completes the proof.
\end{proof}

\section{Conclusions and future research}
\label{S7}

In this work we studied the in-out-proper orientation number of graphs. We proved that for any graph $G$,  $ \overleftrightarrow{\chi}(G) \leq \Delta(G)$. 
We conjectured that there exists a constant number  $c$ such that for every planar graph $G$, we have $\overleftrightarrow{\chi}(G) \leq c $. Regarding this conjecture, we showed that for every tree $T$ we have $\overleftrightarrow{\chi}(T) \leq 3 $ and this bound is sharp.
It is interesting to prove constant bounds for other families of planar graphs.

We also studied the in-out-proper orientation number of subcubic graphs. By using the properties of totally unimodular matrices
we proved that there is a polynomial time algorithm to determine whether $\overleftrightarrow{\chi}(G) \leq 2$, for a given graph $G$ with maximum degree three.  It is interesting to present a polynomial time algorithm for other families of graphs.

It is also interesting to characterize all graphs $G$ which satisfy $\overrightarrow{\chi} (G) =\overleftrightarrow{\chi}(G) $. It would be interesting to attack this problem for the family of regular graphs.

\section{Acknowledgments}

The author would like to thank the anonymous referees for their useful comments which helped to improve the presentation of this paper.

\small
\bibliographystyle{plain}
\bibliography{In-out-proper}

\end{document}